\documentclass[oneside,english]{amsart}
\usepackage[T1]{fontenc}
\usepackage[latin9]{inputenc}
\usepackage{amsthm}
\usepackage{amstext}
\usepackage{amssymb}
\usepackage{esint}

\makeatletter
\numberwithin{equation}{section}
\numberwithin{figure}{section}
\theoremstyle{plain}
\newtheorem{thm}{\protect\theoremname}[section]

\theoremstyle{plain}
\theoremstyle{definition}

\theoremstyle{plain}
\newtheorem{lem}[thm]{\protect\lemmaname}
\newtheorem{cor}[thm]{\protect\corollaryname}
\theoremstyle{plain}

\theoremstyle{plain}
\makeatother

\usepackage{babel}
\providecommand{\definitionname}{Definition}
\providecommand{\lemmaname}{Lemma}
\providecommand{\theoremname}{Theorem}
\providecommand{\corollaryname}{Corollary}
\providecommand{\remarkname}{Remark}
\providecommand{\propositionname}{Proposition}

\DeclareMathOperator{\cp}{cap}

\begin{document}

\title[Extension operators on Sobolev spaces with  decreasing integrability]
{Extension operators on Sobolev spaces with  decreasing integrability}

\author{A.~Ukhlov}
\begin{abstract}
We study extension operators on Sobolev spaces with decreasing integrability on the base of set functions associated with the operator norms. Sharp necessary conditions in the terms of the generalized density condition and the terms of weak equivalence of Euclidean and intrinsic metrics are given. 
\end{abstract}
\maketitle
\footnotetext{\textbf{Key words and phrases:} Sobolev spaces, Extension operators} \footnotetext{\textbf{2010
Mathematics Subject Classification:} 46E35}

\section{Introduction }

Let $\Omega$ be a domain in the Euclidean space $\mathbb R^n$, $n\geq 2$. Recall that the operator
$$
E: W^1_p(\Omega)\to W^1_q(\mathbb R^n),\,\,1\leq q\leq p\leq\infty,
$$
is called an extension operator on Sobolev spaces (with decreasing integrability in the case $q<p$), if $E(f)\big|_{\Omega}=f$ for any function $f\in W^1_p(\Omega)$ and 
$$
\|E\|=\sup\limits_{f\in W^1_p(\Omega)}\frac{\|E(f)\mid W^1_q(\mathbb R^n)\|}{\|f\mid W^1_p(\Omega) \|}<\infty.
$$

Sobolev extension operators arise in the analysis of PDE (see, for example, \cite{M,S70}) and play an important role in the Sobolev spaces theory. In the present article we prove the sharp Ahlfors type necessary generalized ($p,q$)-measure density condition for extension operators of seminormed Sobolev spaces: {\it Let there exists a continuous linear extension operator $E: L^1_p(\Omega)\to L^1_q(\mathbb R^n)$, $n<q\leq p<\infty$,  then
\begin{equation}
\label{density}
\Phi(B(x,r))^{p-q}|B(x,r)\cap\Omega|^{q}\geq c_0 |B(x,r)|^p, \,\,0<r< 1,
\end{equation}
where $\Phi$ is an additive set function associated with the extension operator norm and a constant $c_0=c_0(p,q,n)$ depends on $p$, $q$ and $n$ only.} In the case $p=q$ the measure density condition was introduced in \cite{HKT} (see, also \cite{V87}) and the study of the case $q<p$ requires set functions associated with the extension operators \cite{U99,VU04}. 

It is well known \cite{C61, S70} that if $\Omega\subset\mathbb R^n$ be a Lipschitz domain, then there exists the bounded extension operator $E: W^1_p(\Omega)\to W^1_p(\mathbb R^n)$, $1\leq p\leq \infty$. In \cite{J81} the notion of $(\varepsilon,\delta)$-domains was introduced and it was proved that in every $(\varepsilon,\delta)$-domain there exists the bounded extension operator $E: W^k_p(\Omega)\to W^k_p(\mathbb R^n)$, for all $k\geq 1$  and $p\geq 1$.

The complete description of extension operators of the homogeneous Sobolev space $L^1_2(\Omega)$, $\Omega\subset\mathbb R^2$, were obtained in \cite{VGL79} in the terms of quasi-hyperbolic (quasiconformal) geometry of domains. Namely, it was proved that a simply connected domain  $\Omega\subset\mathbb R^2$ is the $L^1_2$-extension domain iff $\Omega$ is an Ahlfors domain (quasi-disc). In the case of spaces $L^k_p(\Omega)$, $2<p<\infty$, defined in domains $\Omega\subset\mathbb R^2$, necessary and sufficient conditions were obtained in \cite{ShZ16} and the conditions were formulated in the terms of sub-hyperbolic metrics. Note, that extension operators on Sobolev spaces $W^k_p(\Omega)$ were intensively studied in the last decade, see, for example, \cite{FIL14,HKT,K98,KYZ10,Sh17}, but the problem of the complete characterization of Sobolev extension domains in the general case is open. 

If $p>n$ necessary conditions on $W^1_p$-extension domains in the terms of an intrinsic metric and a measure density were obtained in \cite{V87}. Extension operators of Sobolev spaces defined in domains of Carnot groups $E: W^1_p(\Omega)\to W^1_p(\mathbb G)$ were considered in \cite{G96} and extensions of Sobolev spaces on metric measure spaces in \cite{HKT2}.

Results of \cite{HKT,V87} state that there no exists the extension operator 
$$
E: W^1_p(\Omega)\to W^1_p(\mathbb R^n),\,\,1\leq p<\infty,
$$
in H\"older cusp domains $\Omega\subset\mathbb R^n$. In \cite{GS82}, using the method of reflections, were constructed extension operators with decreasing integrability from domains with the H\"older cusps. Later, the more general theory of composition operators on Sobolev spaces with decreasing integrability was founded in \cite{U93,VU02}. Using another technique extension operators in such type domains were considered in \cite{MP86,MP87}. In \cite{B76} was studied an extension operator from H\"older singular domains with decreasing smoothness. The detailed study of extension operators on Sobolev spaces defined in non-Lipschitz domains was given in \cite{MP07}.

Extension operators with decreasing integrability was considered in \cite{U99} where was first introduced a set function (measure) associated with extension operators and were obtained necessary conditions in integral terms. In the present article we give sharp  necessary condition of existence of extension operators on Sobolev spaces with  decreasing integrability in the capacitary terms and we prove the generalized ($p,q$)-measure density condition that refines results of \cite{U99} and generalized \cite{HKT} in the case $n<q<p<\infty$.

Necessary conditions in the terms of intrinsic metrics were considered also. On this base lower estimates of norms of extension operators were obtained. Norm estimates of extension operators have applications in the spectral theory of non-linear elliptic operators and give estimates of Neumann eigenvalues in the terms of operator's norms \cite{GPU20}.

\section{Set functions associated with extension operator }

Let $\Omega$ be a domain in the Euclidean space $\mathbb R^n$, $n\geq 2$, then the Sobolev space $W^1_p(\Omega)$, $1\leq p\leq\infty$, is defined 
as a Banach space of locally integrable weakly differentiable functions
$f:\Omega\to\mathbb{R}$ equipped with the following norm: 
\[
\|f\mid W^1_p(\Omega)\|=\| f\mid L_p(\Omega)\|+\|\nabla f\mid L_p(\Omega)\|,
\]
where $\nabla f$ is the weak gradient of the function $f$. The homogeneous seminormed Sobolev space $L^1_p(\Omega)$, $1\leq p\leq\infty$,  is considered with the seminorm: 
\[
\|f\mid L^1_p(\Omega)\|=\|\nabla f\mid L_p(\Omega)\|.
\]
We consider the Sobolev spaces as Banach spaces of equivalence classes of functions up to a set of $p$-capacity zero \cite{M}. 

\subsection{Set functions and capacity }

Let $A\subset\mathbb R^n$ be an open bounded set such that $A\cap\Omega\ne \emptyset$. Denote by $W_0(A;\Omega)$ the class of continuous  functions $f\in L^1_p(\Omega)$ such that $f\eta$ belongs to $L^1_p(A\cap\Omega)\cap C_0(A\cap\Omega)$ for all smooth functions $\eta\in C_0^{\infty}(\Omega)$. We define the set function
$$
\Phi(A)=\sup\limits_{f\in W_0(A;\Omega)}\left(\frac{\|E(f)\mid L^1_q(A)\|}{\|f\mid L^1_p(A\cap\Omega) \|}\right)^{\kappa},\,\,
\frac{1}{\kappa}=\frac{1}{q}-\frac{1}{p}>0.
$$
This set function was introduced in \cite{U99} in connection with the lower estimates of norms of extension operators of Sobolev spaces. For readers convenience we give the detailed proof of the following theorem announced in \cite{U99}: 

\begin{thm}
Let there exists a a continuous linear extension operator 
$$E: L^1_p(\Omega)\to L^1_q(\mathbb R^n),\,\, 1\leq q< p<\infty.$$ 
Then the function $\Phi(A)$ be a bounded monotone countably additive set function defined on open bounded subsets $A\subset\mathbb R^n$.
\end{thm}

\begin{proof}
Let $A_1\subset A_2$ be open subsets of $\mathbb R^n$. Then extending functions of $C_0(A_1)$ by zero we have that $C_0(A_1)\subset C_0(A_2)$ and we obtain
\begin{multline*}
\Phi(A_1)=\sup\limits_{f\in W_0(A_1;\Omega)}\left(\frac{\|E(f)\mid L^1_q(A_1)\|}{\|f\mid L^1_p(A_1\cap\Omega) \|}\right)^{\kappa}
\leq \sup\limits_{f\in W_0(A_1;\Omega)}\left(\frac{\|E(f)\mid L^1_q(A_2)\|}{\|f\mid L^1_p(A_2\cap\Omega) \|}\right)^{\kappa}\\
\leq
\sup\limits_{f\in W_0(A_2;\Omega)}\left(\frac{\|E(f)\mid L^1_q(A_2)\|}{\|f\mid L^1_p(A_2\cap\Omega) \|}\right)^{\kappa}
=\Phi(A_2).
\end{multline*}
Hence $\Phi$ is the monotone set function. 

Consider open bounded disjoint sets $A_k$, $k=1,2,...$ such that $A_0=\bigcup_{k=1}^{\infty} A_k$. We choose arbitrary functions 
$f_k\in W_0(A_k;\Omega)$ such that
\begin{multline*}
\|E(f_k)\mid L^1_q(A_k)\|\geq \left(\Phi(A_k)\left(1-\frac{\varepsilon}{2^k}\right)\right)^{\frac{1}{\kappa}}\|f_k\mid L^1_p(A_k\cap\Omega)\|,\\
\|f_k\mid L^1_p(A_k\cap\Omega)\|^p=\Phi(A_k)\left(1-\frac{\varepsilon}{2^k}\right),
\end{multline*}
where $k=1,2,...$ and $\varepsilon\in (0,1)$ is a fixed number. Setting $g_N= \sum\limits_{k=1}^N f_k$ we find
\begin{multline*}
\left\|E(g_N)\bigg\vert L^1_q\left(\bigcup_{k=1}^N A_k\right)\right\|\geq 
\left(\sum\limits_{k=1}^N\left(\Phi(A_k)\left(1-\frac{\varepsilon}{2^k}\right)\right)^{\frac{q}{\kappa}}\|g_N\mid L^1_p(A_k\cap\Omega)\|^q\right)^{\frac{1}{q}}\\
=
\left(\sum\limits_{k=1}^N\Phi(A_k)\left(1-\frac{\varepsilon}{2^k}\right)\right)^{\frac{1}{\kappa}}\left\|g_N\bigg\vert L^1_p\left(\left(\bigcup_{k=1}^N A_k\right)\cap\Omega\right)\right\|\\
\geq
\left(\sum\limits_{k=1}^N\Phi(A_k)-\varepsilon\Phi(A_0)\right)^{\frac{1}{\kappa}}\left\|g_N\bigg\vert L^1_p\left(\left(\bigcup_{k=1}^N A_k\right)\cap\Omega\right)\right\|,
\end{multline*}
since the sets where $\nabla E(f_k)$ do not vanish are disjoint. By the last inequality,
we have
$$
\Phi(A_0)^{\frac{1}{\kappa}}\geq \sup\frac{\left\|E(g_N)\bigg\vert L^1_q\left(\bigcup_{k=1}^N A_k\right)\right\|}{\left\|g_N\bigg\vert L^1_p\left(\left(\bigcup_{k=1}^N A_k\right)\cap\Omega\right)\right\|}\geq
\left(\sum\limits_{k=1}^N\Phi(A_k)-\varepsilon\Phi(A_0)\right)^{\frac{1}{\kappa}},
$$
where the upper bound is taken over all above functions $g_N\in W_{0}\left(\left(\bigcup_{k=1}^N A_k\right);\Omega\right)$.
Since both $N$ and $\varepsilon$ are arbitrary, we have
$$
\sum\limits_{k=1}^{\infty}\Phi(A_k)\leq \Phi\left(\bigcup_{k=1}^{\infty}A_k\right).
$$
The inverse inequality can be proved directly.
\end{proof}

\begin{cor}
Let there exists a continuous linear extension operator 
$$
E: L^1_p(\Omega)\to L^1_q(\mathbb R^n),\,\, 1\leq q< p<\infty.
$$
Then
\begin{equation}
\label{ineq}
\|E(f)\mid L^1_q(A)\|\leq \Phi(A)^{\frac{1}{\kappa}}\|f\mid L^1_p(A\cap\Omega)\|,\,\,
\frac{1}{\kappa}=\frac{1}{q}-\frac{1}{p},
\end{equation}
for any function $f\in W^1_{\infty}(A)\cap C_0(A)$.
\end{cor}

Recall the notion of a variational $p$-capacity \cite{GResh}. The condenser in the domain $\Omega\subset \mathbb R^n$ is the pair $(F_0,F_1)$ of connected closed relatively to $\Omega$ sets $F_0,F_1\subset \Omega$. A continuous function $f\in L_p^1(\Omega)$ is called an admissible function for the condenser $(F_0,F_1)$,
if the set $F_i\cap \Omega$ is contained in some connected component of the set $\operatorname{Int}\{x\in\Omega : f(x)=i\}$,\ $i=0,1$. We call as the $p$-capacity of the condenser $(F_0,F_1)$ relatively to domain $\Omega$
the following quantity:
\begin{equation}
\label{cap}
{{\cp}}_p(F_0,F_1;\Omega)=\inf\|f\vert L_p^1(\Omega)\|^p.
\end{equation}
Here the greatest lower bond is taken over all functions admissible for the condenser $(F_0,F_1)\subset\Omega$. If the condenser has no admissible functions we put the capacity equals to infinity.

Let $F_1=E$ subset of open set $U\subset\Omega$ and $F_0=\Omega\setminus U$, then the condenser $R=(E,U)=(\Omega\setminus U,E)$ is called a ring condenser or ring. Note, that the infimum in (\ref{cap}) can be taken on over functions $f\in C^{\infty}_0(\Omega)$ such that $f=1$ on $E$ and $f=0$ on $\Omega\setminus U$.

\begin{thm}
Let there exists a a continuous linear extension operator 
$$
E: L^1_p(\Omega)\to L^1_q(\mathbb R^n),\,\, 1\leq q< p<\infty.
$$
Then for any compact set
$E\subset (U\cap\Omega)$ the inequality
\begin{equation}
\label{ineqcp}
\cp^{\frac{1}{q}}_q(E,U)\leq \Phi(U)^{\frac{1}{\kappa}}\cp^{\frac{1}{p}}_p(E,(U\cap\Omega)),\,\,
\frac{1}{\kappa}=\frac{1}{q}-\frac{1}{p},
\end{equation}
holds for any open set $U\subset\mathbb R^n$.
\end{thm}

\begin{proof}
Let a smooth function $u\in L^1_p(\Omega)$ be an admissible function for the condenser $(E,(U\cap\Omega))\subset\Omega$. Then, extending $u$ by zero on the set $U\setminus \Omega$ we obtain the function $E(u)\in L^1_q(\mathbb R^n)$ which be an admissible function for the condenser $(E,U)\subset\mathbb R^n$. Hence, by the inequality (\ref{ineq}) we have
$$
\cp^{\frac{1}{q}}_q(E,U)\leq \Phi(U)^{\frac{1}{\kappa}}\|u\mid L^1_p(\Omega)\|.
$$
Since $u$ is arbitrary admissible function for the condenser $(E,(U\cap\Omega))\subset\Omega$, then
$$
\cp^{\frac{1}{q}}_q(E,U)\leq \Phi(U)^{\frac{1}{\kappa}}\cp^{\frac{1}{p}}_p(E,(U\cap\Omega)).
$$
\end{proof}

\subsection{Generalized $(p,q)$-measure density conditions}

Consider measure density conditions in domains allow extension operators with decreasing integrability.

\begin{thm}
\label{main1}
Let there exists a a continuous linear extension operator 
$$
E: L^1_p(\Omega)\to L^1_q(\mathbb R^n),\,\, n<q< p<\infty.
$$
Then the domain $\Omega$ satisfies the  generalized $(p,q)$-measure density condition:
\begin{equation*}
\Phi(B(x,r))^{p-q}|B(x,r)\cap\Omega|^{q}\geq c_0 |B(x,r)|^p, \,\,0<r< 1,
\end{equation*}
where $x\in \overline{\Omega}$ and a constant $c_0=c_0(p,q,n)$ depends on $p$, $q$ and $n$ only.
\end{thm}

\begin{proof}
Fix a smooth test function $\eta:\mathbb R^n\to\mathbb R$, with $\operatorname{supp}(\eta)\subset B(0,1)$, such that  $\eta$ is equal to $1$ in the neighborhood of $0\in\mathbb R^n$ and $0\leq\eta(x)\leq 1$ for all $x\in\mathbb R^n$. Consider points $x\in\overline{\Omega}$, $y\in\Omega$, and denote by $r:=|x-y|$. Then the function
$$
f(z)=\eta\left(\frac{x-z}{r}\right)
$$
be a smooth function such that $f=1$ in the neighborhood of $x\in\overline{\Omega}$, $f(y)=0$ and 
$$
|\nabla f(z)|\leq \frac{\widetilde{C}}{r}\,\,\, \text{for all}\,\,\, z\in\mathbb R^n.
$$
Substituting this test function $f$ in the inequality (\ref{ineq}) we obtain
\begin{multline*}
\|f\mid L^1_q( B(x,r)\|\leq \Phi(B(x,r))^{\frac{p-q}{pq}}\|f\mid L^1_p(B(x,r)\cap\Omega)\|\\
\leq
\Phi(B(x,r))^{\frac{p-q}{pq}}\frac{\widetilde{C}}{r}|B(x,r)\cap\Omega|^{\frac{1}{p}}.
\end{multline*}
Because $q>n$, applying the embedding theorem of the Sobolev spaces into the space of H\"older continuous functions (see, for example, \cite{M})
$$
L^1_q(B(x,r))\hookrightarrow H^{\gamma}(B(x,r)),\,\, \gamma=1-n/q, 
$$
we have
$$
\frac{1}{|x-y|^{1-\frac{n}{q}}}=\frac{|f(x)-f(y)|}{|x-y|^{1-\frac{n}{q}}}\leq \|f\mid H^{\gamma} B(x,r)\|\leq C \|f\mid L^1_q( B(x,r)\|.
$$
So, using these inequalities we obtain
$$
\frac{\left(r^n\right)^{\frac{1}{q}}}{r}=\frac{1}{|x-y|^{1-\frac{n}{q}}}\leq \Phi(B(x,r))^{\frac{p-q}{pq}} C \frac{\widetilde{C} }{r}|B(x,r)\cap\Omega|^{\frac{1}{p}}.
$$
Hence
$$
\left(r^n\right)^{\frac{1}{q}}\leq \Phi(B(x,r))^{\frac{p-q}{pq}} C \widetilde{C} |B(x,r)\cap\Omega|^{\frac{1}{p}}.
$$
and the required inequality is proved.
\end{proof}

To prove sharpness of the condition (\ref{density}) we consider as an example the H\"older singular domain $\Omega_{\alpha}$, $\alpha>1$, \cite{GS82,KZ18,MP07}:
$$
\Omega_{\alpha}=\left\{(x_1,x_2)\in\mathbb R^2: 0<x_1\leq 1, |x_2|<x_1^{\alpha} \right\}\cup B((2,0),\sqrt{2}).
$$
Then $|B(0,r)\cap\Omega_{\alpha}|=cr^{\alpha+1}$ and substituting it in the inequality (\ref{density}) we obtain
$$
\Phi(B(0,r))^{p-q}r^{(\alpha+1)q}\geq C r^{2p}, \,\,0<r< 1.
$$
Hence $1\leq q<2p/(\alpha+1)$ that coincide with the sufficient condition of existence of $(p,q)$-extension operators \cite{GS82,KZ18,MP07}. So, the necessary condition of Theorem~\ref{main1} is sharp.

\vskip 0.2cm

Recall that a bounded domain $\Omega\subset\mathbb R^n$ is called $\alpha$-integral regular domains \cite{U99} if the function 
$$
K(x)=\limsup\limits_{r\to 0}\frac{|B(x,r)|}{|B(x,r)\cap\Omega|}
$$
belongs to the Lebesgue spaces $L_{\alpha}(\overline{\Omega})$. By Theorem~\ref{main1} follows the assertion which was originally formulated in \cite{U99}:

\begin{thm}
Let there exists a a continuous linear extension operator 
$$
E: L^1_p(\Omega)\to L^1_q(\mathbb R^n),\,\, n<q< p<\infty.
$$
Then the domain $\Omega$ be $\alpha$-integral regular for $\alpha=q/(p-q)$ and 
$$
\|E\|\geq c_1 \|K \mid L_{\alpha}(\overline{\Omega})\|^{\frac{1}{p}},
$$
where a constant $c_1=c_1(p,q,n)$ depends on $p$, $q$ and $n$ only.
\end{thm}

\begin{proof}
Rewrite the inequality (\ref{density}) in the form
$$
\left(\frac{|B(x,r)|}{|B(x,r)\cap\Omega|}\right)^{\frac{q}{p-q}}\leq \frac{1}{c_1^{\kappa}}\frac{\Phi(B(x,r))}{|B(x,r)|}.
$$
Putting $r\to 0$ and using the Lebesgue type differentiability theorem we have
$$
K(x)^{\alpha}\leq \frac{1}{c_1^{\kappa}}\Phi'(x),\,\,\text{for almost all}\,\,x\in\overline{\Omega}.
$$
Integrating the last inequality on the closed domain $\overline{\Omega}$ we obtain that for any bounded open set $\overline{\Omega}\subset U\subset\mathbb R^n$
$$
\int\limits_{\overline{\Omega}}K(x)^{\alpha}~dx\leq \frac{1}{c_1^{\kappa}}\int\limits_{\overline{\Omega}}\Phi'(x)~dx\leq \frac{1}{c_1^{\kappa}}\int\limits_{U}\Phi'(x)~dx=\frac{1}{c_1^{\kappa}} \Phi(U)\leq \frac{1}{c_1^{\kappa}} \|E\|^{\kappa}.
$$
\end{proof}

\subsection{Intrinsic metrics in extension domains}

Let $\gamma:[a,b]\to\Omega$ be a rectifiable curve, then the length $l(\gamma)$ can be calculated by the formula
$$
l(\gamma)=\int\limits_a^b {\left\langle \dot{\gamma}(t),\dot{\gamma}(t)\right\rangle}^{\frac{1}{2}}~dt.
$$
We define in the domain $\Omega\subset\mathbb R^n$ an intrinsic metric in Alexandrov's sense \cite{A48}:
$$
d_{\Omega}(x,y)=\inf l(\gamma(x,y)), \,\,x,y\in\Omega,
$$
where infimum is taken over all rectifiable curves $\gamma\subset\Omega$ joint points $x,y\in\Omega$.

We use the following lemma \cite{G96,V88}. 

\begin{lem}
\label{func}
For any points $x,y\in\Omega$ there exists a function $f\in W^1_{\infty}(\Omega)$ such that:
\begin{enumerate}
	\item $0\leq f(t)\leq 1$ for any $t\in\Omega$, $f(x)=1$ and $f(y)=0$,
	\item $|f(t)-f(s)|\leq d_{\Omega}(t,s)/d_{\Omega}(x,y)$,
	\item $\operatorname{supp} (f)\subset B(x, d_{\Omega}(x,y)):=B(x,R)$, 
	\item $|\nabla f|\leq 1/d_{\Omega}(x,y)$ a.e. in $\Omega$.
\end{enumerate}
 \end{lem}

Note, that the proof of this lemma is based on the test function 
$$
f(t)=\frac{d_{\Omega}(t,\Omega_x)}{d_{\Omega}(x,y)},\,\,t\in\Omega, 
$$
for fixed $x,y\in\Omega,$ which was introduced in \cite{V88} (see, also \cite{G96}). The sets $\Omega_x$ and $d_{\Omega}(t,\Omega_x)$ are defined for fixed points $x,y\in\Omega$ by the formulas
$$
\Omega_x=\{s\in\Omega : d_{\Omega}(x,s)\geq d_{\Omega}(x,y)\}
$$
and
$$
d_{\Omega}(t,\Omega_x)=\inf\{d_{\Omega}(t,s) : s\in\Omega_x\}.
$$
.

In the following theorem we give relation between the intrinsic metric and the Euclidean metric in $(p,q)$-extension domains that can be consider as a generalized Ahlfors type metric condition.

\begin{thm}
\label{main2}
Let there exists a a continuous linear extension operator 
$$
E: L^1_p(\Omega)\to L^1_q(\mathbb R^n),\,\, n<q\leq p<\infty.
$$
Then in the domain $\Omega$ intrinsic metric is $(p,q)$-equivalent to the Euclidean metric:
\begin{equation}
\label{eqpq}
d_{\Omega}(x,y)^{1-\frac{n}{p}}\leq C_0\Phi(B(x,R))^{\frac{1}{\kappa}}|x-y|^{1-\frac{n}{q}},\,\,R=d_\Omega(x,y),
\end{equation}
for all $|x-y|<1$, where a constant $C_0=C_0(p,q,n)$ depends on $p$, $q$ and $n$ only.
\end{thm}

\begin{proof}
Substituting the test function $f$ from Lemma~\ref{func} in the inequality (\ref{ineq}) we obtain
\begin{equation}
\label{ineq3}
\|E(f)\mid L^1_q(B(x,R)\|\leq \Phi(B(x,R))^{\frac{1}{\kappa}}\cdot \frac{1}{d_\Omega(x,y)^{1-\frac{n}{p}}},
\end{equation}
because 
\begin{multline*}
\|f\mid L^1_p(\Omega)\|\leq \left(\int\limits_{B(x,R)}|\nabla f(z)|^p~dz\right)^{\frac{1}{p}}\\
\leq 
\left(\int\limits_{B(x,R)}\left(\frac{1}{R}\right)^p~dz\right)^{\frac{1}{p}}=\frac{1}{R^{1-\frac{n}{p}}}=\frac{1}{d_\Omega(x,y)^{1-\frac{n}{p}}}.
\end{multline*}

In the left side of the inequality~(\ref{ineq3}) we apply the embedding theorem of the Sobolev spaces into the space of H\"older continuous functions $L^1_q(B)\hookrightarrow H^{\gamma}(B)$, $\gamma=1-n/q$. So, we obtain
$$
\frac{1}{|x-y|^{1-\frac{n}{q}}}=\frac{|f(x)-f(y)|}{|x-y|^{1-\frac{n}{q}}}\leq \|f\mid H^{\gamma} (B(x,R))\|\leq C_0 \|f\mid L^1_q(B(x,R)\|.
$$

Hence
$$
\frac{1}{|x-y|^{1-\frac{n}{q}}}\leq C_0\Phi(B(x,R))^{\frac{1}{\kappa}}\cdot \frac{1}{d_\Omega(x,y)^{1-\frac{n}{p}}},\,\,|x-y|<1.
$$
The theorem proved.
\end{proof}

Let $x\in\Omega$, we define the value \cite{U99}
$$
M(x)=\limsup_{r\to 0}M(x,r):=\limsup_{r\to 0}\left\{\inf\limits_{|x-y|\leq r}\left\{m: d_{\Omega(x,y)}\leq m |x-y|\right\}\right\}.
$$

The inequality (\ref{eqpq}) leads to the following lower estimate of the extension operator formulated in \cite{U99}:

\begin{thm}
\label{main4}
Let there exists a a continuous linear extension operator 
$$
E: L^1_p(\Omega)\to L^1_q(\mathbb R^n),\,\, n<q< p<\infty.
$$
Then 
\begin{equation}
\label{est}
\|E\|\geq C_0 \|M\mid L_{\alpha}(\Omega)\|^{1-\frac{n}{q}},
\end{equation}
where $\alpha=(pq-pn)/(p-q)$ and a constant $C_0=C_0(p,q,n)$ depends on $p$, $q$ and $n$ only.
\end{thm}

\vskip 0.5cm

Department of Mathematics, Ben-Gurion University of the Negev, P.O.Box 653, Beer Sheva, 8410501, Israel 
							
\emph{E-mail address:} \email{ukhlov@math.bgu.ac.il


\begin{thebibliography}{References}

\bibitem{A48} A.~D.~Alexandrov, Intrinsic geometry of convex surfaces, Gostechizdat, Moscow, 1948.

\bibitem{B76} V.~I.~Burenkov, A way of continuing differentiable functions, Studies in the theory of differentiable functions of several variables and its applications, VI. Trudy Mat. Inst. Steklov. 140 (1976), 27--67. 

\bibitem{C61}	A.~P.~Calder\'on, Lebesgue spaces of differentiable functions and distributions, 
Proc. Symp. Pure Math., 4 (1961), 33--49.

\bibitem{FIL14} C.~Fefferman, A.~Israel, G.~Luli, Sobolev extension by linear operators,
J. Amer. Math. Soc., 27 (2014), 69--145.

\bibitem{GPU20} V.~Gol'dshtein, V.~Pchelintsev, A.~Ukhlov, Sobolev extension operators and Neumann eigenvalues, 
J. Spectr. Theory, 10 (2020), 337--353. 

\bibitem{GResh} V.~M.~Gol'dshtein, Yu.~G.~Reshetnyak, Quasiconformal mappings and Sobolev spaces, Dordrecht,
Boston, London: Kluwer Academic Publishers, 1990.

\bibitem{GS82} V.~M.~Gol'dshteĭn, V.~N.~Sitnikov, Continuation of functions of the class $W^1_p$ across H\"older boundaries, Imbedding theorems and their applications, Trudy Sem. S. L. Soboleva, (1982), 31--43.

\bibitem{G96} A.~V.~Greshnov, Extension of differentiable functions beyond the boundary of the domain on Carnot groups,Sobolev spaces and related problems of analysis, Trudy Inst. Mat., 31 (1996) 161--186. 

\bibitem{HKT}	P.~Hajlasz, P.~Koskela, H.~Tuominen, Sobolev embeddings, extensions
and measure density condition, J. Funct. Anal., 254 (2008), 1217--1234.

\bibitem{HKT2}	P.~Hajlasz, P.~Koskela, H.~Tuominen, Measure density and extendability of Sobolev functions,
 Rev. Mat. Iberoam., 24 (2008), 645--669.

\bibitem{J81} P.~W.~Jones, Quasiconformal mappings and extendability of functions in Sobolev spaces, 
Acta Math., 147 (1981) 71--78.

\bibitem{K98} P.~Koskela, Extensions and imbeddings, J. Funct. Anal., 159 (1998), 369--383.

\bibitem{KYZ10} P.~Koskela, D.~Yang, Y.~Zhou, A Jordan Sobolev extension domain, Ann. Acad. Sci. Fenn. Math., 35 (2010), 309--320.

\bibitem{KZ18} P.~Koskela, Zh.~Zhu, Sobolev extension via reflections,  	arXiv:1812.09037 (2018).

\bibitem{M} V.~Maz'ya, Sobolev spaces: with applications to elliptic
partial differential equations, Springer, Berlin/Heidelberg, 2010.

\bibitem{MP86} V.~G.~Maz'ya, S.~V.~Poborchi, Extension of functions in Sobolev classes to
the exterior of a domain with the vertex of a peak on the boundary I, Czech. Math. Journ., (1986) 634--661.

\bibitem{MP87} V.~G.~Maz'ya, S.~V.~Poborchi, Extension of functions in Sobolev classes to
the exterior of a domain with the vertex of a peak on the boundary II, Czech. Math. Journ., (1987) 128--150.

\bibitem{MP07} V.~Maz'ya, S.~Poborchi, Embedding and extension theorems for functions in non-Lipschitz domains, Petersburg: St. Petersburg Publishers, Russia, 2007.

\bibitem{Sh17} P.~Shvartsman, Whitney-type extension theorems for jets generated by Sobolev functions, 
Adv. Math., 313 (2017), 379--469.

\bibitem{ShZ16} P.~Shvartsman, N.~Zobin, On planar Sobolev $L^m_p$-extension domains, Adv. Math., 287 (2016), 237--346. 

\bibitem{S70} E.~M.~Stein, Singular integrals and differentiability properties of functions, Princeton University Press,
Princeton, New Jersey, 1970.

\bibitem{U93} A.~Ukhlov, On mappings, which induce embeddings of
Sobolev spaces, Siberian Math. J., 34 (1993), 185--192.

\bibitem{U99} A.~Ukhlov, Lower estimates for the norm of the extension operator of the weak differentiable functions on domains of Carnot groups. Proc. of the Khabarovsk State Univ., 8 (1999), 33--44.

\bibitem{V87} S.~K.~Vodop'yanov, Geometric properties of domains and mappings. Lower bounds on the norm of the extension operator,
Trudy Inst. Mat.,  7 (1987), 70--101.

\bibitem{V88} S.~K.~Vodop'yanov, Taylor's formula and functional spaces, Novosibirsk State Univ., Novosibirsk, 1988.

\bibitem{VGL79} S.~K.~Vodop'yanov, V.~M.~Gol'dstein, T.~G.~Latfullin, A criterion for the extension of functions of the class $L^1_2$ from unbounded plane domains, Siberian. Math. J., 20 (1979), 416--419.

\bibitem{VU02} S.~K.~Vodop'yanov, A.~D.~Ukhlov, Superposition operators in Sobolev spaces, Russian Mathematics (Izvestiya VUZ) 46
(2002), no. 4, 11--33. 

\bibitem{VU04}	S.~K.~Vodop'yanov, A.~Ukhlov, Set functions and their applications in the theory of Lebesgue and Sobolev spaces. II, Matem. trudy, 7 (2004), 13--49. 

\end{thebibliography}
\end{document}